\documentclass[10pt]{article}
\usepackage[cp1251]{inputenc}
\usepackage{graphicx}
\usepackage{amsmath,amssymb,amsopn,eufrak, amsthm}
\newtheorem{theorem}{Theorem}[section]
\newtheorem{proposition}{Proposition}[section]
\newtheorem{corollary}{Corollary}[section]

\begin{document}

\begin{center}
\Large

{\bf On varieties of rings whose finite rings are determined by their zero-divisor graphs}
\end{center}

\begin{center}
{\bf A.S. Kuzmina}

{\it Altai State Pedagogical Academy, 55, Molodezhnaya st.,\\
Barnaul, 656031, Russian\\akuzmina1@yandex.ru}
\end{center}

\begin{center}
{\bf Yu.N. Maltsev}

{\it Altai State Pedagogical Academy, 55, Molodezhnaya st.,\\
Barnaul, 656031, Russian\\maltsevyn@gmail.com}
\end{center}

\begin{abstract}
The zero-divisor graph $\Gamma(R)$ of an associative ring $R$ is the graph whose vertices are all nonzero zero-divisors (one-sided and two-sided)  of  $R$, and two distinct vertices $x$ and $y$ are joined by an edge iff either $xy=0$ or $yx=0$.

In the present paper, we study some properties of ring varieties where every finite ring is uniquely determined by its zero-divisor graph.

\vspace{0.2cm}


\vspace{0.2cm}


\end{abstract}

\section{Introduction}	

Throughout this paper, any ring $R$ is associative (not necessarily commutative).

The zero-divisor graph $\Gamma(R)$ of a ring $R$ is the graph whose vertices are all nonzero zero-divisors (one-sided and two-sided)  of  $R$, and two distinct vertices $x$ and $y$ are joined by an edge iff either $xy=0$ or $yx=0$ \cite{Redmond}.

The notion of the zero-divisor graph of a commutative ring was introduced by I.Beck in \cite{Beck}. In this paper, all elements of a ring are vertices of the graph. In \cite{AL}, D.F.Anderson and P.S.Livingston introduced the zero-divisor graph whose vertices are nonzero zero-divisors of a ring.  In  \cite{AL}, the authors studied the interplay between the ring-theoretic properties of a commutative ring $R$ with unity  and the graph-theory properties of $\Gamma(R)$. For a noncommutative ring the definition of zero-divisor graph was intoduced in \cite{Redmond}.

The question of when $\Gamma(R)\cong \Gamma(S)$ implies that $R\cong S$ is very interesting. For finite reduced rings and finite local rings this question has been investigated in \cite{AM-274}. (We note that ring $R$ is called \textit{reduced} if $R$ has no nonzero nilpotent elements.)  In this paper, we study varieties of rings, where  $\Gamma(R)\cong\Gamma(S)$ implies  $R\cong S$ for all finite rings $R,S$.  Note that some results concerning such varieties have been proved in~\cite{semr}.

\vspace{0.2cm}

Firstly, we fix some notations. Let $\mathbb{Z}$ be the set of integers, $\mathbb{N}$ the set of natural numbers, $\mathbb{Z}[x]$ the polynomial ring over $\mathbb{Z}$.   For each prime number $p$ by $GF(p^n)$ we denote the Galois field with $p^n$ elements. For each number $n$ let $\mathbb{Z}_{n}$ be the residue-class ring modulo $n$. The symbol $J(R)$ denotes the Jacobson radical of a ring $R$. We define a finite ring $R$ with unity to be an {\em local ring} if the factor-ring $R/J(R)$ is a field.

For each prime number $p$ let

$N_{0,p^n}=\left\langle a; a^2=0, p^na=0\right\rangle;~~$ $N_{p^2}=\left\langle a; a^2=pa,p^2a=0\right\rangle;$

$N_{p,p}=\left\{ \left( \begin{array}[c]{ccc} 0 & a & b \\ 0 & 0 & a \\ 0 & 0 & 0\end{array} \right) ; a, b \in GF(p)\right\};$ 

$A_{p}= \left( \begin{array}[c]{cc} GF(p) & GF(p)  \\ 0 & 0 \end{array} \right)$;  $A_{p}^0= \left( \begin{array}[c]{cc} GF(p) & 0  \\ GF(p) & 0 \end{array} \right)$. 

Let the additive group of a ring $R$ be a direct sum of its nonzero subgroups $A_i,$ $i=1,\ldots,n$ and $n\geq 2,$ i.e. $R=A_1\stackrel{.}{+}\ldots \stackrel{.}{+} A_n.$ If  $A_i$ is ideal of $R$ for all $i$, then we say that the ring $R$ is  {\it decomposable} and write ${R=A_1 \oplus \ldots \oplus A_n}$. 
A ring $R$ is called \textit{ subdirectly irreducible} if the intersection of all its nonzero ideals  is a nonzero ideal of $R$ \cite{Jac}. It is known that every ring is a subdirect sum of subdirectly irreducible rings \cite{Jac}. 
The ring of $n\times n$ matrices over a ring $R$ is denoted by $M_n(R).$ For all elements $x,y$ of a ring $R$ we put $[x,y]=xy-yx$.  

For every set $X=\{x_1,x_2,\ldots\}$ let $\mathbb{Z}\left\langle X\right\rangle=\mathbb{Z}\left\langle x_1,x_2,\ldots\right\rangle$ be the free associative ring freely generated by the set $X$. For every $f(x_1,\ldots,x_d)\in \mathbb{Z}\left\langle X\right\rangle$ the number $$min\{deg(h)~\left|\mbox{~all nonzero monomials~} h \mbox{~of~} f\right.\}$$  is called {\it the lower degree} of the polynomial $f(x_1,\ldots,x_d)$. We say that an polynomial $f(x_1,\ldots,x_d)$  is {\it essentially depending } on $x_1,x_2,\ldots, x_d$ if  $f(0,x_2,\ldots,x_d)=\ldots=$ $=~f(x_1,\ldots,x_{d-1},0)=0. $

Let $\mathfrak{M}$ be a variety of associative rings. We denote by $T(\mathfrak{M})$ the $T$-ideal of all polynomial identities of $\mathfrak{M}.$ For a set $\{f_i\left| i\in I\right.\}\subseteq \mathbb{Z}\left\langle X\right\rangle$ by $\{f_i~~|~~i\in I\}^T$ denote the smallest $T$-ideal containing all $f_i.$  Also, let $T(R)$ be the $T$-ideal of all polynomial identities satisfied by a ring $R.$ 

For all varieties $\mathfrak{M}$ and $\mathfrak{N}$ by $\mathfrak{M}\vee \mathfrak{N}$ denote the union of  these varieties. Note that $T\left(\mathfrak{M}\vee \mathfrak{N}\right)=T\left(\mathfrak{M}\right)\cap T \left(\mathfrak{N}\right).$ Put ${\mathfrak{L}_{p_1,\ldots, p_s}=var~N_{0,p_1}\vee \ldots \vee var~N_{0,p_s},}$ where $p_1, \ldots, p_s$ are  prime numbers such that $p_i\neq p_j$ for $i\neq j$. 

A finite ring $R$ is called {\it critical}, if it does not belong to the variety generated by all its proper subrings and factor-rings \cite{LvovI}. We say that a variety $\mathfrak{M}$ is \textit{Cross} if the following conditions hold: (i)~ all rings of $\mathfrak{M}$ are locally finite; (ii)~ the set of all critical rings in $\mathfrak{M}$ is finite; (iii)~ $T(\mathfrak{M})$ has  a finite basis. By \cite{LvovI}, a variety of associative rings is Cross if and only if  it is generated by a finite associative ring. 

In this article, $K_n$ will denote the complete graph on $n$ vertices.

\vspace{0.2cm}

The aim of this paper is to prove the following theorems.

\begin{theorem}\label{th1} Suppose $\mathfrak{M}$ is a variety of associative rings such  that ${xy+f(x,y)\in T(\mathfrak{M})}$, where the lower degree of $f(x,y)$ is greater then $2$. Then $\Gamma(R)\cong\Gamma(S)$ implies $R\cong S$ for all finite rings $R,S\in \mathfrak{M}$  if and only if $\mathfrak{M}\subseteq \mathfrak{L}_{p_1,\ldots, p_s}\vee var~\mathbb{Z}_p$ where $p, p_1, \ldots, p_s$ are prime numbers and $(p_i,p)\neq (3,2)$ for $i\leq s.$\end{theorem}

\begin{theorem} \label{th2}Let $\mathfrak{M}$ be a variety of associative rings such that $\mathbb{Z}_p\in \mathfrak{M}$ for some prime number $p$. Suppose  $\Gamma(R)\cong\Gamma(S)$  implies $R\cong S$ for all finite rings $R,S\in \mathfrak{M}$. Then any subdirectly irreducible finite ring $A\in \mathfrak{M}$ satisfies one of the following conditions:

(1) $A\cong\mathbb{Z}_p;$

(2) $A$ is a nilpotent ring and $q^2A=(0)$ for some prime number $q$.\end{theorem}

\begin{theorem}\label{th3} Let $\mathfrak{M}$ be a variety of associative rings such  that $\mathbb{Z}_2\in \mathfrak{M}$ and $\Gamma(R)\cong\Gamma(S)$ implies $R\cong S$ for all finite rings $R,S\in \mathfrak{M}.$ Then any subdirectly irreducible finite ring $A\in \mathfrak{M}$ of order $2^t$ ($t>0$) satisfies one of the following conditions:

(1) $A\cong\mathbb{Z}_2;$

(2) $A$ is a nilpotent commutative ring and $2x=0,$ $x^2=0$ for each $x\in A.$ \end{theorem}

Note that Theorem~\ref{th2} strengthens the main result of \cite{semr}.

\section{The auxiliary results}

To prove the main theorems, we need several supplementary results. Propositions~\ref{pr1}--\ref{pr5} were proved in~\cite{semr}. These statements will be used in what follows.

By Tarski's theorem (see \cite{Tarski}),  any nontrivial variety of rings contains either $var~\mathbb{Z}_p$, or  $var~N_{0,p},$ where $p$ is some prime number. The question of when ${\Gamma(R)\cong \Gamma(S)}$ implies  $R\cong S$ for all finite rings $R,S$ in $var~\mathbb{Z}_p$ ($var~N_{0,p}$) is interesting. The following proposition answers this question.

\begin{proposition}  \label{pr1} Suppose $\mathfrak{M}$ is either $var~\mathbb{Z}_p$, or  $var~N_{0,p},$ where $p$ is some prime number.  Then $\Gamma(R)\cong \Gamma(S)$ implies  $R\cong S$  for all finite rings $R,S\in \mathfrak{M}$ (see \cite{semr}).\end{proposition}

\begin{proposition}\label{pr2} Let $\mathfrak{M}$ be a variety of rings such that $\Gamma(R)\cong \Gamma(S)$ implies  $R\cong S$  for all finite rings $R,S\in \mathfrak{M}$. Then the following conditions hold: \begin{enumerate}
	\item  if $\mathbb{Z}_p\in \mathfrak{M}$ for some prime number~$p$,  then $\mathfrak{M}$ does not contain any field with the exception of~$\mathbb{Z}_p$;
	\item either $x^t\in T(\mathfrak{M}),$ where $t>0,$ or $\mathbb{Z}_p\in \mathfrak{M}$ for some prime number $p$;
	 \item  if a local ring $R$ is in $\mathfrak{M}$, then it is a field;
	 \item if $n\geq 2$, then $N_{0,p^n}\notin \mathfrak{M}$ for each prime number~$p$ (see \cite{semr}).
\end{enumerate}\end{proposition}

\begin{proposition}\label{pr3} Let $p_1, \ldots, p_s$ be  prime numbers such that $p_i\neq p_j$ for $i\neq j$.  Then $\Gamma(R)\cong \Gamma(S)$ implies  $R\cong S$ for all finite rings $R,S\in \mathfrak{L}_{p_1,\ldots, p_s}$   (see \cite{semr}).\end{proposition}

\begin{corollary}\label{c1} Let $R$ be a finite ring. Then ${\Gamma(R)=K_2}$ iff $R$ is isomorphic to one of the following rings: $$N_{0,3}, \mathbb{Z}_9, \mathbb{Z}_3[x]/(x^2), \mathbb{Z}_2\oplus\mathbb{Z}_2 \mbox{~(see \cite{semr}).}$$ \end{corollary}

\begin{proposition}\label{pr4} Suppose  $\mathfrak{M}=\mathfrak{L}_{p_1,\ldots, p_s}\vee var~\mathbb{Z}_p,$ where $p,p_1, \ldots, p_s$ are  prime numbers such that $p_i\neq p_j$ for $i\neq j$ ($p$ may be equal to $p_i$ for some $i$).  Then $\Gamma(R)\cong \Gamma(S)$ implies  $R\cong S$ for all finite rings $R,S\in \mathfrak{M}$ iff  $(p_i,p)\neq(3,2)$ for~$i\leq s$ (see \cite{semr}). \end{proposition}

\begin{proposition}\label{pr5}  For every prime number~$p$  $$\Gamma\left(N_{p^2}\right)=\Gamma\left(N_{p,p}\right)=\Gamma\left(A_p\right)=\Gamma\left(A^0_p\right)=\Gamma\left(N_{0,p}\oplus \mathbb{Z}_p\right)$$ (see \cite{semr}).\end{proposition}

\begin{corollary}\label{c2} Suppose $\mathfrak{M}$ is a variety of rings such that $\Gamma(R)\cong \Gamma(S)$ implies  $R\cong S$ for all finite rings $R,S\in \mathfrak{M}$. Then $A_p, A^0_p \notin \mathfrak{M}$ for each prime number~$p$ (see \cite{semr}).\end{corollary}

\begin{corollary}\label{c3} Suppose $\mathfrak{M}$ is a variety of rings such that $\Gamma(R)\cong \Gamma(S)$ implies  $R\cong S$ for all finite rings $R,S\in \mathfrak{M}$. If  $\mathbb{Z}_p\in \mathfrak{M}$ for some prime number~$p$, then $N_{p,p} \notin \mathfrak{M}$ (see \cite{semr}).\end{corollary}

\begin{corollary}\label{c4}  Suppose $\mathfrak{M}$ is a variety of rings such that $\Gamma(R)\cong \Gamma(S)$ implies  $R\cong S$ for all finite rings $R,S\in \mathfrak{M}$. Then $N_{q^2}\notin \mathfrak{M}$ for each odd prime number~$q$ (see \cite{semr}).\end{corollary}

Before proving the main results, let us prove a number of supplementary results.

\begin{proposition}\label{pr7} Suppose $\mathfrak{M}$ is a variety of rings such that $\Gamma(R)\cong \Gamma(S)$ implies  $R\cong S$ for all finite rings $R,S\in \mathfrak{M}$. Then  $mx, d x+ x^2 g(x)\in T(\mathfrak{M}),$  where $m\in\mathbb{N},$ $g(x)\in \mathbb{Z}[x]$,   either $d=1$, or $d=q_1 q_2\ldots q_l,$ and $q_1, q_2, \ldots, q_l$ are mutually different prime divisors of $m$.  \end{proposition}

\begin{proof} By Proposition~\ref{pr2}(4) we have $N_{0,4}\notin\mathfrak{M}.$ Note that  $T(N_{0,4})=\{4x, xy\}^T.$ Therefore $T(\mathfrak{M})\not\subseteq \{xy\}^T$. It means that $T(\mathfrak{M})$ contains a polynomial $f(x)=$ $=kx+x^2\varphi(x),$ where $k$ is some nonzero integer and $\varphi(x)$ is some polynomial in $\mathbb{Z}[x].$  Let $f(x)=kx+x^2\varphi(x)=kx+a_2x^2+a_3x^3+\ldots+a_sx^s,$ where $a_2,a_3, \ldots, a_s\in \mathbb{Z},$ $s\geq 2$. Then
\begin{equation*}\begin{split}
2^s&f(x)-f(2x)=\\&=(2^s-2)kx+(2^s-2^2)a_2x^2+(2^s-2^3)a_3x^3+\ldots+(2^s-2^{s-1})a_{s-1}x^{s-1}=0.
\end{split}\end{equation*} 
Repeating the argument, we see that $$(2^s-2)(2^{s-1}-2)\ldots(2^2-2)kx\in T(\mathfrak{M}).$$ 

Thus the variety $\mathfrak{M}$ satisfies the identities  $mx=0$ and  $kx+x^2\varphi(x)=0,$ where ${m=(2^s-2)(2^{s-1}-2)}$ ${\ldots(2^2-2)k}$. We may assume that ${k\geq 1.}$ Let $m=q_1^{\beta_1} q_2^{\beta_2}\ldots q_t^{\beta_t},$ where $q_1,q_2,\ldots ,q_t$ are prime numbers such that $q_i\neq q_j$ for $i\neq j.$ By Proposition~\ref{pr2}(4), we have $N_{0,q_i^2}\notin\mathfrak{M}$ for $i\leq t.$ Therefore  $T(\mathfrak{M})\not\subseteq T(N_{0,q_i^2})=\{q_i^2x, xy\}^T,$ $i\leq t.$ This implies that for every $i\leq t$ there exists  a polynomial $\alpha_ix+x^2\psi_i(x)$ in $T(\mathfrak{M})$  such that $\alpha_i\in \mathbb{Z}$ and $q_i^2$ is not a divisor of~$\alpha_i$.  Let $d$ be the greatest common divisor of the numbers $\alpha_1, \alpha_2, \ldots, \alpha_t, m.$ We see that either $d=1$, or $d=q_{i_1}q_{i_2}\ldots q_{i_l},$ where $q_1, q_2, \ldots, q_l$ are mutually different prime divisors of $m$.   If $d=1$, then the proof is straightforward. Now let $d=q_{i_1}q_{i_2}\ldots q_{i_l}\neq 1.$ Note that  $q_{i_\mu}\neq q_{i_\nu}$  for $\mu\neq \nu.$ Further there exist integers $v_1, v_2, \ldots, v_t, v$ such that $$\alpha_1v_1+ \alpha_2v_2+ \ldots+ \alpha_tv_t+ mv=d.$$ Multiplying the identities $$\alpha_1x+x^2\psi_1(x)=0,\ldots,\alpha_tx+x^2\psi_t(x)=0, mx=0$$ by $v_1, v_2, \ldots, v_t, v$ respectively and summing them, we obtain  ${dx+x^2g(x) \in T(\mathfrak{M})}$ for some $g(x)\in \mathbb{Z}[x].$ Finally, notice that every prime divisor of $d$ is a divisor of $m$. This completes the proof.
\end{proof}

\begin{proposition}\label{pr8}Suppose $\mathfrak{M}$ is a variety of rings such that $\Gamma(R)\cong \Gamma(S)$ implies  $R\cong S$ for all finite rings $R,S\in \mathfrak{M}$. Then $T(\mathfrak{M})$ contains an identity of the form~$mx$, where $m=q_1^{\beta_1} q_2^{\beta_2}\ldots q_t^{\beta_t},$ $\beta_i\leq 3$ for all $i\leq t,$ and $q_1,q_2,\ldots ,q_t$ are prime numbers such that $q_i\neq q_j$ for $i\neq j.$\end{proposition}

\begin{proof} From Proposition~\ref{pr7}, $T(\mathfrak{M})$ contains an identity of the form~$mx$ for some integer~$m$. Let $F$ be the one-generated free ring in $\mathfrak{M}.$ Suppose ${m=q_1^{\beta_1} q_2^{\beta_2}\ldots q_t^{\beta_t},}$ where $q_1,q_2,\ldots ,q_t$ are prime numbers such that $q_i\neq q_j$ for $i\neq j.$ Therefore $F=$ $=\oplus_{i=1}^{t}A_i,$ where $A_i$ is a ideal of $F$ and $q_i^{\beta_i}A_i=(0)$ for $i\leq t.$ 

We shall show that $\beta_i\leq 3$ for each $i.$  Assume the contrary. Then we can assume without loss of generality that $\beta_1\geq 4.$ Since $2\left[\dfrac{\beta_1}{2}\right]+2\geq \beta_1,$ we have $$\left(q_1^{\left[\frac{\beta_1}{2}\right]+1}A_1\right)^2=q_1^{2\left[\frac{\beta_1}{2}\right]+2}A_1^2=(0).$$ If the abelian group  $\langle q_1^{\left[\frac{\beta_1}{2}\right]+1}A_1, +\rangle$ contains a element of additive order $q_1^\delta$ for some $\delta\geq 2,$ then the ring $A_1$ has a subring $S$ such that $S\cong N_{0, q_1^2}.$ This contradicts Proposition~\ref{pr2}(4). Hence $q_1\left(q_1^{\left[\frac{\beta_1}{2}\right]+1}A_1\right)=(0)$ and $\left[\dfrac{\beta_1}{2}\right]+2\geq \beta_1.$ 
If $\beta_1=2a+1$ for some natural number $a$,  then we get $a+2\geq 2a+1.$ Hence $a\leq 1$ and $\beta_1\leq 3.$ Now assume that  $\beta_1=2a$ for some natural number $a$. Therefore  ${\left(q_1^aA_1\right)^2=(0)}$.   As before, it can be shown  that ${q_1\left(q_1^aA_1\right)=(0).}$ Thus ${a+1\geq \beta_1,}$ i.e. $a\leq 1$ and $\beta_1\leq 2.$ So we have proved that $\beta_1\leq 3.$ This contradiction concludes the proof.
\end{proof}

\begin{proposition}\label{pr9} Suppose $\mathfrak{M}$ is a variety of rings such that $\Gamma(R)\cong \Gamma(S)$ implies  $R\cong S$ for all finite rings $R,S\in \mathfrak{M}$. For any finite ring $R\in\mathfrak{M}$ there exist prime numbers $q_1,q_2,\ldots ,q_t$ such that $q_1^2 q_2^2\ldots q_t^2R=(0)$ and  $q_i\neq q_j$ for $i\neq j$.\end{proposition}

\begin{proof} It follows from Propositions~\ref{pr7} and~\ref{pr8} that  $T(\mathfrak{M})$  contains the polynomials $q_1^{\beta_1} q_2^{\beta_2}\ldots q_t^{\beta_t}x$ and $ q_{i_1} q_{i_2}\ldots q_{i_l} x+ x^2 g(x),$ where $g(x)\in\mathbb{Z}[x],$ $q_1,q_2,\ldots ,q_t$ are prime numbers, $q_i\neq q_j$ for $i\neq j$, and   $\beta_i\leq 3$ for each $i$. Therefore  $q_1^3 q_2^3\ldots q_t^3R=(0).$ We see that $R=\oplus_{i=1}^{t}R_i,$ where $q_i^3R_i=(0)$ for all $i\leq t.$ Now let us prove that $q_1^2R_1=(0).$ 

Assume that $q_1\notin\{q_{i_1}, q_{i_2},\ldots, q_{i_l}\}$. In this case, there exist integers $a,b$ such that $q_1^3a+ q_{i_1} q_{i_2}\ldots q_{i_l}b=1.$ Hence $R_1$ satisfies the identity $$q_1^3ax+  q_{i_1} q_{i_2}\ldots q_{i_l}bx+bx^2g(x)=x+x^2h(x)=0.$$ Thus ${R_1\cong \mathbb{Z}_{q_1}\oplus\ldots\oplus \mathbb{Z}_{q_1}}$ and $q_1R_1=(0)$ (see \cite{Jac}).  

Now assume that $q_1\in\{ q_{i_1}, q_{i_2},\ldots, q_{i_l} \}$. In the same way, it can be proved that $R_1$ satisfies a identity $f_1(x)=q_1x+x^2g_1(x)=0,$ where $g_1(x)\in \mathbb{Z}[x].$ For any nilpotent element $a\in R_1$ it follows that $$0=f_1(q_1a)=q_1^2a+q_1^2a^2g_2(a)=q_1^2a(1+ag_2(a)),$$ where $g_2(a)=g_1(q_1a).$ Since the element $a$ is nilpotent,  we get $q_1^2a=0.$ So for the case $R_1=J(R_1),$ we have $q_1^2R_1=(0).$ Now we can assume that ${R_1\neq J(R_1).}$ In this case, there exists a nonzero idempotent  $e\in R_1$ such that $e+J(R)$ is a unity in  the factor-ring $R_1/J(R_1)$ (see \cite[p. 80, 94]{E}). Therefore $$R_1=eR_1e\stackrel{.}{+} eR_1(1-e)\stackrel{.}{+} (1-e)R_1e\stackrel{.}{+} (1-e)R_1(1-e)$$ (see \cite[p. 32]{E}). Also, note that $eR_1(1-e)\stackrel{.}{+} (1-e)R_1e\stackrel{.}{+} (1-e)R_1(1-e)\subseteq J(R_1)$ and ${q_1^2J(R_1)=(0)}.$ We shall show that $q_1^2e=0.$ By Wilson's theorem (see \cite{W}), it follows that $eR_1e=Q\stackrel{.}{+} N,$ where $Q$ is a direct sum of matrix rings over Galois rings,  $N$ is a $(Q,Q)$-bimodule such that $N\subseteq J(R_1)$. Let $Q=\oplus_{i=1}^{m}M_{k_i}(S_i),$ where $S_i$ is a Galois ring for all $i\leq m.$ Assume that $k_1\geq 2.$ In this case, the variety~$\mathfrak{M}$ contains the ring~$A_{q_1}.$ This contradicts Corollary~\ref{c2}. Therefore $k_1=1.$ Similarly, it can be proved that ${k_2=\ldots=k_m=1.}$ This means that $Q=\oplus_{i=1}^{m}S_i.$ It is known that every Galois ring is local. From Proposition~\ref{pr2}(3), we have that  $S_i$ is a field for each $i$. Thus $q_1Q=(0).$ It implies that $q_1e=0.$ So  $q_1^2R_1=(0)$. In the same way, we can prove that $q_i^2R_i=(0)$ for $i\geq 2.$ It shows that $q_1^2 q_2^2\ldots q_t^2R=(0).$ This completes the proof.
\end{proof}

\section{The proofs of main results}

Now we are in a position to prove our main theorems.

\vspace{0.5cm}

\textbf{The proof of Theorem~\ref{th1}. }

Suppose $\mathfrak{M}\subseteq\mathfrak{L}_{p_1,\ldots, p_s}\vee var~\mathbb{Z}_p$  and $(p_i,p)\neq(3,2)$ for each $i\geq 1$. From Proposition~\ref{pr4}, $\Gamma(R)\cong \Gamma(S)$ implies  $R\cong S$ for all finite rings $R,S\in \mathfrak{M}$. Moreover, $x(1-x^{p-1})y\in T(\mathfrak{M}),$ i.e.  $\mathfrak{M}$ satisfies the identity $xy-x^py=0.$

Conversely, suppose $\Gamma(R)\cong \Gamma(S)$ implies  $R\cong S$ for all finite rings ${R,S\in \mathfrak{M}}$ and  $xy+f(x,y)\in T(\mathfrak{M})$, where the lower degree of $f(x,y)$ is greater then~$2$. By Proposition~\ref{pr7}, we have that $T(\mathfrak{M})$ contains a polynomials $mx, q_1 q_2\ldots q_l x+ x^2 g(x),$ where $m\in\mathbb{N},$ $g(x)\in \mathbb{Z}[x]$, $q_1,q_2,\ldots ,q_l$ are prime numbers such that $q_i\neq q_j$ for $i\neq j.$ From Lvov's theorem (see \cite{LvovI}), $\mathfrak{M}$ is a Cross variety. Therefore it is generated by its critical rings. 

Consider a critical ring $R\in\mathfrak{M}$. From Propositions~\ref{pr7} and~\ref{pr9}, the ring~$R$ satisfies a identities $q_1 q_2\ldots q_l x+ x^2 g(x)$ and $q^2x=0,$ where $g(x)\in \mathbb{Z}[x]$, $q, q_1, q_2, \ldots ,q_l$ are prime numbers such that $q_i\neq q_j$ for $i\neq j.$  Hence for some $h(x)\in \mathbb{Z}[x]$  either  $x+x^2h(x)\in T(R),$  or $qx+x^2h(x)\in T(R)$.  In the first case, we have $R\cong \mathbb{Z}_{q}$ (see~\cite{Jac}). Now we can assume that the ring~$R$ satisfies the identity $qx+x^2h(x)=0.$ Let us consider the following cases.

\textbf{Case 1:} $R=J(R).$ In this case, from the identity $xy+f(x,y)=0$, we get ${R^2=(0).}$ Since $qx+x^2h(x)=0$ is a identity of $R$, we have $qx=0$ for each $x\in R.$ Thus $R\in var~N_{0,q}.$

\textbf{Case 2:}  $J(R)=(0).$ From the Wedderburn~-- Artin theorem (see \cite[p. 80]{E}) and Corollary~\ref{c2}, it follows that $R\cong \mathbb{Z}_q.$ Thus $R\in var~\mathbb{Z}_q.$

\textbf{Case 3:} $(0)\neq J(R) \neq R.$  As above (see Case~1), we have ${qJ(R)=(0).}$ Let $e^2=e$ be an idempotent of the ring $R$ such that $e+J(R)$ is a unity in the factor-ring $R/J(R).$ As before (see the proof of Proposition~\ref{pr9}), we have $qe=0$. So the ring  $R$ is a $\mathbb{Z}_q$~- algebra. It means that $R$ is isomorphic to one of the following algebras: $\left( \begin{array}[c]{cc} GF(q_1) & GF(q_2)  \\ 0 & 0 \end{array} \right),$  $\left( \begin{array}[c]{cc} GF(q_1) & 0  \\ GF(q_2) & 0 \end{array} \right)$, ${\left\{\left( \begin{array}[c]{cc} a & b  \\ 0 & \sigma(a) \end{array} \right);a,b\in GF(q^n)\right\}},$ $\left( \begin{array}[c]{cc} GF(q_1) & GF(q_3)  \\  0& GF(q_2) \end{array} \right)$, where  $GF(q_1)\subseteq GF(q_3), $ $GF(q_2)\subseteq GF(q_3)$ and $\sigma$ is an automorphism  of the field $GF(q^n)$ such that $\sigma\neq 1$. Hence $\mathfrak{M}$ contains either $A_q $ ($A_q^0$) or a local ring. This contradicts Proposition~\ref{pr2}(3) and Corollary~\ref{c2}. Thus Case~3 is impossible.

Theorem~\ref{th1} is proved.$\Box$

\vspace{0.5cm}

\textbf{The proof of Theorem~\ref{th2}. }

Let $\mathfrak{M}$ be a variety of rings such that the following conditions hold: (i)~${\Gamma(R_1)\cong \Gamma(R_2)}$ implies  $R_1\cong R_2$ for all finite rings $R_1,R_2\in \mathfrak{M}$; (ii)~$\mathbb{Z}_p\in\mathfrak{M}$ for some prime number $p$. Consider a subdirectly irreducible finite ring $R$ in $\mathfrak{M}$.  From Theorem of~\cite{semr}, we have either $R\cong \mathbb{Z}_p,$ or $R=J(R).$ If $R=J(R),$ then, from Proposition~\ref{pr9}, it follows that $q^2R=(0)$ for some prime number $q$. 

The theorem is proved. $\Box$

\vspace{0.5cm}

\textbf{The proof of Theorem~\ref{th3}. }

Let $\mathfrak{M}$ be a variety of associative rings such  that the following conditions hold: (i)~$\mathbb{Z}_2\in \mathfrak{M}$; (ii)~$\Gamma(R_1)\cong\Gamma(R_2)$ implies $R_1\cong R_2$ for all finite rings $R_1,R_2\in \mathfrak{M}.$  Suppose  $R\in \mathfrak{M}$ is a subdirectly irreducible finite nonzero ring  of order $2^t$. From Theorem~\ref{th2}, it follows that either $R\cong \mathbb{Z}_2,$ or $R^n=(0)$  and $p^2R=(0)$ for some numbers $n>1$ and $p$ ($p$ is prime). If $R\cong \mathbb{Z}_2,$ then the proof is trivial. Assume that $R^n=(0)$ and $p^2R=(0)$. Since $\left|R\right|=2^t,$ we have $p=2$ and $4R=(0).$

We shall show that $N_4\notin\mathfrak{M}.$  Assume the contrary. Then $N_4\in\mathfrak{M}$. Therefore $N_{0,2}\oplus \mathbb{Z}_2 \in \mathfrak{M}.$ From Proposition~\ref{pr5}, we have $\Gamma(N_4)=\Gamma(N_{0,2}\oplus \mathbb{Z}_2).$ By assumption, $\Gamma(R_1)\cong\Gamma(R_2)$ implies $R_1\cong R_2$ for all finite rings $R_1,R_2\in \mathfrak{M}.$ So we have a contradiction. Hence $N_4\notin\mathfrak{M}.$ This yields that there exists a polynomial $f(x_1, \ldots, x_d)$ such that $f(x_1, \ldots, x_d)$ is essentially depending on $x_1,x_2,\ldots, x_d$ and ${f(x_1, \ldots, x_d)\in T(\mathfrak{M})\setminus T(N_4).}$ We note that $T(N_4)=\{xyz, 4x, 2xy, 2x+x^2\}^T$. Therefore $d\leq 2.$ Let us consider two cases.

\textbf{Case 1:} $d=2.$ 

We can assume that the polynomial $f(x,y)$ has a form $$f(x,y)=xy+\alpha[x,y]+2\psi(x,y)+\varphi(x,y),$$ where $\alpha\in\mathbb{Z},$ $\psi(x,y), \varphi(x,y)\in\mathbb{Z}\left\langle x,y\right\rangle$ and the lower degree of $\varphi(x,y)$ is greater then $2$.  Substituting $y$ for $x$ in  $f(x,y)$, we obtain $f(x,x)=x^2+2\beta x^2+x^3\varphi_1(x)$ for some $\varphi_1(x)\in\mathbb{Z}[x].$ Clearly, $x^2+2\beta x^2+x^3\varphi_1(x)\in T(\mathfrak{M}).$  So the ring $R$ satisfies the identities $4x=0$ and $(1+2\beta) x^2+x^3\varphi_1(x)=0.$ Further, there exist integers $u,v$ such that $(1+2\beta)u+4v=1.$ Therefore the ring $R$ satisfies $x^2=x^3\varphi_2(x)$ for some $\varphi_2(x)\in\mathbb{Z}[x].$  Since $R^n=(0),$   $x^2=0$ for each $x\in R$. From proposition~\ref{pr7}, it follows that $T(\mathfrak{M})$ contains a polynomial of the form $$q_1q_2\ldots q_sx+x^2g(x),$$ where $g(x)\in\mathbb{Z}[x]$, $q_1,q_2,\ldots, q_s$ are prime numbers, and $q_1\neq q_j$ for $i\neq j.$ Assume that the numbers $q_1,q_2,\ldots, q_s$ are odd. In this case, there exist integers $q,t$ such that $(q_1q_2\ldots q_s)q+4t=1.$ Since the ring $R$ satisfies $q_1q_2\ldots q_sx+x^2g(x)=0$ and $4x=0,$ the $T$-ideal $T(R)$ contains a polynomial $x-x^2g_1(x)$ for some $g_1(x)\in\mathbb{Z}[x].$ It is clear that $R=(0)$. We have a contradiction.  Thus $q_i=2$ for some $i.$  As above, it can proved that  $R$ satisfies $2x+x^2g_2(x)=0$ for some $g_2(x)\in\mathbb{Z}[x].$ Since  $x^2=0$ for every element $x\in R$, the polynomial $2x$ belongs to $T(R).$  Since $x^2=0$ for each $x\in R,$ it is easily shown that the ring $R$ satisfies the identity $xy+yx=0.$ We know that $a=-a$ for each $a\in R$. Therefore $xy-yx=0$ for all $x,y\in R,$ i.e. $R$ is commutative. So $R$ satisfies $x^2=0,$ $xy=yx,$ $x_1 \ldots x_n=0,$ and $2x=0.$

\textbf{Case 2:} $d=1.$ 

In this case, we can assume that the polynomial $f(x,y)$ has a form $$f(x)=\alpha x+\beta x^2+x^3 f_1(x),$$ where $\alpha,\beta\in\mathbb{Z}$ and $f_1(x)\in\mathbb{Z}[x].$ Assume that $\alpha$ is odd. In this case, the ring~$R$ satisfies $x=x^2h(x)$ for some $h(x)\in\mathbb{Z}[x].$ Since $R$ is a nilpotent ring, we see that $R=(0).$ We have a contradiction. Therefore $\alpha$ is even.  Let $\alpha=2m$ for some $m\in\mathbb{N}.$ Hence  \begin{equation}\label{ide} f(x)=2mx+\beta x^2+x^3 f_1(x)=\gamma x^2+m(x^2+2x)+x^3f_1(x),\end{equation} where $\gamma=\beta-m.$ Clearly, $f(x)\in T(N_4)$ whenever   $\gamma$ is even. Since  ${f(x)\not\in T(N_4)}$, it follows that $\gamma$ is odd. Hence there exist integers  $a,b$ such that $\gamma a+4b=1.$ Combining the identities $4x=0$ and~(\ref{ide}),  we see that the ring $R$ satisfies \begin{equation}\label{id} x^2+m_1(x^2+2x)+x^3\mu(x)=0,\end{equation} where $m_1\in \mathbb{Z}$ and  $\mu(x)\in \mathbb{Z}[x].$ 

Assume that $m_1$ is even.  Then  we have $$0=2x^2+2m_1(x^2+2x)+2x^3\mu(x)=2x^2(1+x\mu(x)),$$ because $4x=0$ for each $x\in R.$ Since $R$ is nilpotent, $2x^2=0$  is a identity of $R$.  From~(\ref{id}), it follows that $x^2+x^3\mu(x)=0$ also is a identity of $R$. Hence $x^2=0$ for each $x\in R.$ As above, using Proposition~\ref{pr7}, we can proved that  $R$ satisfies a identity $2x+x^2g_2(x)=0$ for some $g_2(x)\in\mathbb{Z}[x].$ This implies that $2x=0$ for any $x\in R$. So we have proved that  $R$ satisfies $x^2=0,$ $xy=-yx=yx,$ $x_1 \ldots x_n=0,$ and $2x=0$ whenever $m_1$ is even. 

 Now assume that $m_1$ is odd. Let $m=2q+1,$ where $q\in\mathbb{N}$.  In this case,  $T(R)$ contains the identity $$x^2+(2q+1)(x^2+2x)+x^3\mu(x)=0.$$ This identity can be represented in the form $$2(q+1)x^2+2x+x^3\mu(x)=0$$ since $4x=0$ for each $x\in R$. Hence, \begin{equation}\label{id3}2x=-x^3\mu(x)(1+q_1x)^{-1}\end{equation} for each $x\in R$, where $q_1=q+1$. By Corollary~\ref{c3}, we have $N_{2,2}\not\in\mathfrak{M}.$ Consequenly there exists a polynomial $F(x_1,\ldots,x_w)\in T(\mathfrak{M})\setminus T(N_{2,2})$ essentially depending  on $x_1,\ldots,x_w.$ Obviously, $w\leq 2.$ Assume that  $w=2.$ In this case, $F(x,y)$ can be represented in the form $$F(x,y)=xy+\alpha [x,y]+2 \Phi(x,y)+\Psi(x,y),$$  where $\Phi(x,y),\Psi(x,y)\in \mathbb{Z}\left\langle x,y\right\rangle$  and the lower degree of  $\Psi(x,y)$ is greater then~$2$. From~(\ref{id3}), it follows that $$2\Phi(x,y)=-\Phi(x,y)^3\mu(\Phi(x,y))(1+q_1\Phi(x,y))^{-1}.$$  In other words, the lower degree of $2\Phi(x,y)$ is greater then~$2$. Thus $F(x,y)$ can be represented in the form $$F(x,y)=xy+\alpha [x,y]+\Psi'(x,y),$$ where $\Psi'(x,y)\in \mathbb{Z}\left\langle x,y\right\rangle$ and the lower degree of  $\Psi'(x,y)$ is greater then~$2$. Substituting $y$ for $x$ in the identity $F(x,y)=0$, we obtain a identity $x^2=x^3\omega(x)$ for some $\omega(x)\in \mathbb{Z}[x].$ Since $R$ is nilpotent, $x^2=0$ for every $x\in R.$ From~(\ref{id3}),  we get the identity $2x=0$. So $R$ satisfies $x^2=0,$ $xy=yx,$ $x_1 \ldots x_n=0,$ $2x=0$ whenever $w=2.$ Now let us consider the case $w=1.$ The polynomial $F(x,y)$ can be represented in the form $$F(x)=\alpha x+\beta x^2+2 \lambda(x)+x^3p(x),$$ where $\lambda(x),p(x)\in \mathbb{Z}[x],$ $\alpha,\beta\in\{0,1\}$, and one of the numbers $\alpha,\beta$ is not equal to zero. If $\alpha=1$ then $R$ satisfies some identity of the form $$(1+2k)x+x^2\lambda'(x)=0,$$ where $\lambda'(x)\in \mathbb{Z}[x].$ This means that $R=(0).$ We have a contradiction. Therefore $\alpha=0$ and $\beta=1.$ Hence,   $$F(x)=x^2+2\lambda(x)+x^3p(x).$$   Multiplying $F(x)$ by $2$, we get the identity $$2x^2(1+xp(x))=0.$$ We see that $2x^2=0.$ Let $\lambda(x)=a_1x+a_2x^2+\ldots+a_Nx^N$, where $N, a_1, \ldots, a_n$ are some integers. Consequently, $$x^2+2a_1x+x^3p(x)=0$$ is a identity of $R$. If $a_1$ is even then $R$ satisfies $x^2+x^3p(x)=0.$ In this case, $x^2=0,$ $xy=yx,$ $x_1 \ldots x_n=0,$ $2x=0$ are identities of $R$. Now assume that  $a_1$ is odd.  Then $$x^2+2x+x^3p(x)=0$$ for each $x\in R$. From the identity~(\ref{id3}), it follows that $R$ satisfies some identity of the form  $x^2+x^3p_1(x)=0.$ Thus $x^2=0,$ $xy=yx,$ $x_1 \ldots x_n=0,$ $2x=0$ are identities of $R$. 

This completes the proof of Theorem~\ref{th3}.$\Box$

\end{document}